\documentclass[lineno]{biometrika}

\usepackage{tikz-cd}
\usepackage{amsmath}
\usepackage{amsfonts}
\usepackage{mathtools}
\usepackage{hyperref}

\usepackage{newtxtext}
\usepackage[subscriptcorrection]{newtxmath}

\graphicspath{{./art/}}

\usepackage[plain,noend]{algorithm2e}

\makeatletter
\renewcommand{\algocf@captiontext}[2]{#1\algocf@typo. \AlCapFnt{}#2} % text of caption
\def\@algocf@capt@plain{top}
\renewcommand{\algocf@makecaption}[2]{%
  \addtolength{\hsize}{\algomargin}%
  \sbox\@tempboxa{\algocf@captiontext{#1}{#2}}%
  \ifdim\wd\@tempboxa >\hsize%     % if caption is longer than a line
    \hskip .5\algomargin%
    \parbox[t]{\hsize}{\algocf@captiontext{#1}{#2}}% then caption is not centered
  \else%
    \global\@minipagefalse%
    \hbox to\hsize{\box\@tempboxa}% else caption is centered
  \fi%
  \addtolength{\hsize}{-\algomargin}%
}
\makeatother

\newcommand{\bR}{\mathbb{R}}

\newcommand{\fS}{\mathfrak{S}}

\DeclareMathOperator*{\argmin}{arg\,min}

\DeclareMathOperator*{\mini}{minimise}

\DeclareMathOperator{\Skew}{Skew}
\DeclareMathOperator{\Sym}{Sym}

\DeclareMathOperator{\Ortho}{\mathrm{O}}

\DeclareMathOperator{\interior}{int}
\DeclareMathOperator{\grad}{grad}

\nolinenumbers

\begin{document}

% \jname{Biometrika}
%% The year, volume, and number are determined on publication
\jyear{}
\jvol{}
\jnum{}
\cyear{}
%% The \doi{...} and \accessdate commands are used by the production team
%\doi{10.1093/biomet/asm023}
\accessdate{}

%% These dates are usually set by the production team
\received{}
\revised{}

%% The left and right page headers are defined here:
% \markboth{E. Maignant et~al.}{Biometrika style}
\markboth{E. Maignant et~al.}{}

%% Here are the title, author names and addresses
\title{Barycentric subspace analysis of network-valued data}

\author{E. MAIGNANT}
\affil{DTU Compute, Danmarks Tekniske Universitet,\\ Anker Engelunds Vej 101, 2800 Kongens Lyngby, Denmark
\email{elomai@dtu.dk}}

\author{X. PENNEC}
\affil{Epione Project Team, Inria, Université Côte d’Azur,\\ 2004, route des Lucioles, BP 93, 06902 Sophia Antipolis Cedex, France \email{xavier.pennec@inria.fr}}

\author{A. TROUVÉ}
\affil{Centre Borelli, ENS Paris-Saclay, Université Paris-Saclay,\\ 4, avenue des Sciences, 91190 Gif-sur-Yvette, France
\email{alain.trouve@ens-paris-saclay.fr}}

\author{\and A. CALISSANO}
\affil{Department of Statistical Science, University College London,\\ Gower Street, London WC1E 6BT, U.K.
\email{a.calissano@ucl.ac.uk}}

\maketitle

\def\spacingset#1{\renewcommand{\baselinestretch}%
{#1}\small\normalsize} \spacingset{1}

\begin{abstract}
    Certain data are naturally modeled by networks or weighted graphs, be they biological networks or mobility networks. When there is no canonical labeling of the nodes across the dataset, we talk about unlabeled networks. In this paper, we focus on the question of exploratory analysis of this type of data. More specifically, we address the issue of interpreting the feature subspace constructed by dimensionality reduction methods. Most existing methods for network-valued data are derived from principal component analysis (PCA) and therefore rely on subspaces generated by a set of vectors, which we identify as a major limitation in terms of interpretability. Instead, we propose to implement the method called barycentric subspace analysis (BSA), which relies on subspaces generated by a set of points, which we choose, in practice, to be samples. In order to provide a computationally feasible framework for BSA, we introduce a novel embedding for unlabeled networks where we replace their usual representation by equivalence classes of isomorphic networks with that by equivalence classes of cospectral networks. In a simulated study, we demonstrate the improved interpretability of BSA compared to tangent PCA. We then illustrate through two real-world datasets how BSA can be used both to visualize known patterns and to discover new ones in network-valued distributions.
\end{abstract}


\section{Introduction} \label{sec:introduction}

    \subsection{Network-valued data} 
        In recent years, there has been a growing interest in the analysis of network-valued data, namely graphs with scalar weights on the edges, due to the various applications in which they arise. Examples of such data are brain connectivity networks \citep{simpson_permutation_2013, durante_nonparametric_2017, sporns_graph_2018, aeschlimann_inter_2026}, brain arterial networks \citep{guo_quotient_2021}, anatomical trees \citep{wang_object_2007, feragen_tree_2013}, and mobility networks \citep{von_ferber_public_2009}. Different tasks have been addressed for network-valued data such as classification \citep{tsuda_graph_2010, liu_graph_2020}, prediction \citep{calissano_graph_2022, severn_non_2021}, hypothesis testing \citep{simpson_permutation_2013}, data generation \citep{simonovsky_graphvae_2018, vignac_digress_2022}, and dimensionality reduction and visualization \citep{severn_manifold_2022, guo_quotient_2021, calissano_populations_2024}. \\

    \subsection{Principal component analysis versus barycentric subspace analysis} 
        Given the high complexity of network-valued data, initial exploration is often necessary, making dimensionality reduction and visualization crucial tasks. A significant proportion of the works addressing dimensionality reduction -- for network-valued data -- revolve around the extension of principal component analysis (PCA) from tangent approaches \citep{severn_manifold_2022, guo_quotient_2021, calissano_populations_2024}. For interpreting the representation subspace, these methods rely on a set of principal components, that is continuous deformations of a mean network. Such deformations might be however misleading as they tend to span topologies that have no reality with respect to the original data. Typically, they mostly span fully connected topologies, whereas network-valued datasets often feature a few number of less trivial topologies. To address this issue, we propose barycentric subspace analysis (BSA) as an alternative method \citep{pennec_barycentric_2018, rohe_low-dimensional_2018}. The general idea of BSA is the same as PCA: computing a subspace which minimizes the projection error of the dataset. The main difference is that such the feature subspace is formulated as a barycentric subspace, meaning it is is generated by a set of points (called reference points) rather than a set of vectors.
        \begin{figure}[h]
            \figuresize{.37}
            \figurebox{}{1.\textwidth}{}[bsa_euclidean]
            \caption{A schematic visualization of barycentric subspace analysis in the Euclidean case.}
            \label{fig:bsa}
        \end{figure}
        
        \noindent If the reference points are constrained to be selected among the data points (sample-limited BSA), their interpretation is straightforward in light of their original applicative context. Motivated by these properties, the aim of the paper is to extend BSA to the case of a of set of network-valued data and put it forward as a tool for visualization and exploratory data analysis. 

    \subsection{Unlabeled networks}
        When working with a set of network-valued data, there might be some uncertainty around the node labels and the correspondence between nodes across networks. In such cases, the networks are called unlabeled \citep{kolaczyk_averages_2020}. The ambiguity about the correspondence between nodes can vary from a totally unlabeled setting -- two social networks connecting two different groups of individuals -- to a partially labeled setting -- two molecules involving different types of atoms. In this paper, we focus on defining BSA for unlabeled networks. The case of labeled networks can be addressed directly in the Euclidean setting adapting \cite{pennec_barycentric_2018}.
        \smallbreak
        To extend BSA to the case of unlabeled networks, we need a suitable geometric framework. In the literature, the uncertainty around the node labels or node correspondence is often modeled via the action of a permutation group on the nodes of the networks \citep{jain_structure_2009, kolaczyk_averages_2020, calissano_populations_2024}. For example, if weighted undirected networks are represented by their symmetric adjacency matrices, their unlabeled counterpart can be represented in the quotient space of symmetric matrices by the permutation group
        \begin{equation}
            \mathcal{G}_n = \Sym(n) / \fS_n.
        \end{equation}
        However, performing BSA in this discrete setting is challenging as detailed in \cite{calissano_towards_2023}. In  \cite{severn_manifold_2022}, the authors propose another embedding that is inspired by shape analysis \citep{dryden_statistical_2016} and relies on a continuous group acting on networks. The main drawback of this approach is that the analysis of a network viewed as a shape does not always respect the symmetries that occur naturally in the network. 

    \subsection{Our proposal}
        To overcome both the computational limitations of discrete quotient spaces and the difficulties in interpreting the model by \cite{severn_manifold_2022}, we propose a relaxed embedding based on the action of the orthogonal group by conjugation onto the set of weighted adjacency matrices
        \begin{equation}
            \Gamma_n = \Sym(n) / \Ortho(n).
        \end{equation}
        We show that this space is isometric to the space of sorted eigenvalues, automatically providing a full and computationally efficient geometric framework. In particular, we show that the corresponding quotient distance reduces to the well known spectral distance, which is naturally invariant under node relabeling \citep{jurman_introduction_2011,donnat_tracking_2018}. We call $\Gamma_n$ the spectral graph space of size $n$. Along with the distance function, we equip spectral graph spaces with the tools required to perform BSA, namely an exponential map and its inverse.

    %\subsection{Structure of the paper}
       % The paper is organized as follows. In Section \ref{sec:geometry}, we introduce spectral graph spaces and we detail their geometric properties. In Section \ref{sec:bsa}, we formulate barycentric subspace analysis and its sample-limited variant in the context of spectral graph spaces. We run a comparison with principal component analysis in Section \ref{sec:bsa_vs_pca}. To further showcase the applicability of BSA, we study a simulated dataset in Section \ref{appendix:clustered} and a real-world dataset in Section \ref{sec:airlines}. 

\section{Spectral graph spaces} \label{sec:geometry}

    \subsection{Action of the orthogonal group on the set of weighted graphs}
       This section describes the geometric framework of the method. We refer the reader to \cite{lee_introduction_2018} for further background on differential geometry and to \cite{le_riemannian_1993} for a detailed example of the construction of a Riemannian metric on a quotient space. We consider undirected weighted graphs with a finite set of vertices (nodes). Such a graph or network will be given as a triplet $G=(V, E, f)$ consisting of the set $V = \{1, \ldots, n\}$ of nodes, a set $E$ of edges i.e. pairs of nodes, and a scalar function $f$ defined on the edges. Treating non-existing edges as unweighted edges, we can encode the network $G$ by its weighted adjacency matrix $X = (x_{ij}) \in \bR^{n\times n}$ defined by
        \begin{equation}
            x_{ij} = x_{ji} = \left\vert \begin{array}{ll} f(\{i,j\}) & \quad \text{if} \:\: \{i, j\} \in E \vspace{.1cm}\\ 0  & \quad \text{otherwise} \end{array} \right.
        \end{equation}
        such that networks identify with the set $\Sym(n)$ of symmetric matrices of size $n$. Note that we allow for networks to have self-loops, that is for a node to be connected to itself or equivalently for adjacency matrices to have non-zero diagonal entries.
        \smallbreak
        Now in an effort to relax the modeling of unlabeled networks by the action of node permutations, we consider the conjugation action of the orthogonal group $\Ortho(n)$ on the space $\Sym(n)$ of symmetric matrices
        \begin{equation}
            (R, X) \mapsto RXR^T
        \end{equation}
        which transforms one network into another cospectral network, that is a network whose adjacency matrix shares the same spectrum. The orbit of a network consists then exactly of all the other cospectral networks. As the spectrum of a network controls some of its connectivity properties, for example the degree of the nodes \citep{spielman_spectral_2025}, the action of $\Ortho(n)$ tends to identify networks with a similar structure. Finally, because the permutation group is a subgroup of the orthogonal group, the equivalence class of a given network contains in particular all the other networks obtained by permuting its nodes.
        \smallbreak
        Under such a new group action, unlabeled networks are then given as points of the quotient space $\Gamma_n$ defined by
        \begin{equation}
            \Gamma_n = \Sym(n) / \Ortho(n).
        \end{equation}
        It is a stratified differentiable manifold with singularities at the points where the action of the orthogonal group is not free (see \cite{mather_stratifications_1973} for an introduction to stratified spaces and \cite{le_riemannian_1993} for another example of such differentiable manifold). We can show that these singularities correspond exactly to the networks with at least two equal eigenvalues. Let then $\Sym(n)^\ast$ be the subset of symmetric matrices whose eigenvalues are pairwise distinct. On this subset, the action of the orthogonal group is proper and free such that the principal stratum $\Gamma^\ast_n = \Sym(n)^\ast / \Ortho(n)$ is a connected differentiable manifold.
        \begin{definition}
            The quotient space $\Gamma_n$ is called the spectral graph space of size $n$.
        \end{definition}
        Let us now investigate the metric structure of the space. Firstly, we study the stratification as a metric space, and secondly, we focus on the Riemannian structure of the principal stratum. Mainly, we show that spectral graph spaces are isometric to convex polyhedral cones. From this result, we derive in particular that the Riemannian logarithm, central to the definition of the notion of barycenter, consists simply in a translation, and moreover that it extends to the whole stratification.
    
    \subsection{Spectral graph spaces as metric spaces} \label{sec:distance}
        The graph space $\Gamma_n$ is naturally equipped with the quotient distance induced by the $\Ortho(n)$-invariant Frobenius distance on symmetric matrices. More precisely, if $\pi$ denotes the canonical projection, then the function 
        \begin{equation}
            \label{eq:quotient_distance}
            d\left( \pi(X), \pi(Y) \right) = \inf_{R\in \Ortho(n)} \left\|RYR^T - X\right\|_F
        \end{equation}
        is well defined and satisfies all the axioms of a distance function. Such a construction is rather standard in the context of shape analysis \citep{younes_shapes_2010} and makes $(\Gamma_n, d)$  a complete metric space. Since the permutation group is a subgroup of the orthogonal group as we already highlighted, the distance between two networks thus defined is always smaller than their Frobenius distance up to some permutation of their nodes. In other words, two networks with a similar structure have similar equivalence classes. Now we have the following result.
        \begin{proposition}
            \label{prop:distance}
            Let $X, Y \in \Sym(n)$. Then the distance between the two corresponding networks is simply the $\ell^2$ spectral distance
            \begin{equation}
                \label{eq:spectral_distance}
                d\left( \pi(X), \pi(Y) \right) = \bigg( \sum_{r=1}^n \left|\lambda_r(X) - \lambda_r(Y)\right|^2 \bigg)^{1/2}
            \end{equation}
            where $\lambda_1(X) \leq \cdots \leq \lambda_n(X)$ and $\lambda_1(Y) \leq \cdots \leq \lambda_n(Y)$ denote the eigenvalues (with their multiplicity) of $X$ and $Y$ respectively.
        \end{proposition}
        Let us comment on this result. First, it is interesting to note that $\ell^p$ spectral distances are well known distances between networks (see for example \citealp{jurman_introduction_2011, donnat_tracking_2018}). Then, what the proposition above actually tells us is that the map $\lambda = (\lambda_1, \ldots, \lambda_n)$ defined accordingly, and constant over each equivalence class, descends to an isometry between the graph space $\Gamma_n$ and the convex cone of sorted vectors of $\bR^n$ according to the commutative diagram below.
        \begin{equation}
            \label{eq:lambda}
            \begin{tikzcd}
                \Sym(n) \arrow{r}{\lambda} \arrow[swap]{d}{\pi} & \hspace{.1cm} C_n & [-1.1cm] = \{\mu \in \bR^n \: \vert \: \mu_1 \leq \cdots \leq \mu_n \} \\
                [.2cm] \Gamma_n \arrow[swap, dashed]{ur}{\cong}
            \end{tikzcd}.
        \end{equation}
        An important consequence of this statement is that the geodesics (distance-minimizing curves) of the graph space $\Gamma_n$ identify with that of the cone $C_n$, such that there exists a unique minimizing geodesic segment between any two networks described by the straight line segment
        \begin{equation}
            \label{eq:geodesic}
            \gamma(t) = (1-t)\lambda(X) + t\lambda(Y).
        \end{equation}
        For further insight into these different results, we refer to \cite{alekseevsky_riemannian_2003}. In fact, the map $\lambda$ descends to a diffeomorphism between the principal stratum $\Gamma_n^\ast$, which we recall consists of the quotient by the orthogonal group of symmetric matrices with pairwise distinct eigenvalues, and the interior of the open cone $\interior(C_n)$. Let us now extend the diagram \ref{eq:lambda} by equipping the manifold $\Gamma^\ast_n$ with a Riemannian structure compatible with the existing metric structure.
    
    \subsection{The Riemannian structure of spectral graph spaces} \label{sec:metric}
        One way to endow the quotient space $\Gamma_n^\ast$ with a Riemannian metric is to leverage the Riemannian submersion theorem as detailed in \cite{tumpach_three_2023}, following the exact same construction as that developed by \cite{le_riemannian_1993} in the case of shape spaces. Namely, the Riemannian submersion theorem states that there exists a unique Riemannian metric $g$ on $\Gamma_n^\ast$ that makes the surjective map $d\pi$ an isometry between the horizontal subbundle
        \begin{equation}
            H\Sym(n)^\ast = (\ker d\pi )^\perp
        \end{equation}
        and the tangent bundle $T\Gamma_n^\ast$. Moreover, the corresponding Riemannian distance coincides with the quotient distance as introduced in \ref{eq:quotient_distance}.  Let us then construct the metric $g$ explicitly. First, we prove the following. 
        \begin{proposition}
            \label{prop:horizontal}
            Let $X\in \Sym(n)^\ast$. Let us write $X = Q_X \operatorname{diag}(\lambda_1(X), \ldots, \lambda_n(X)) Q_X^T$ for $Q_X \in \Ortho(n)$ unique up to a sign change of each column. Then the horizontal subspace at $X$ is
            \begin{equation}
                H_X\Sym(n)^\ast = \Big\{Q_X \operatorname{diag}(h_1, \ldots, h_n) Q_X^T \: \vert \: (h_1, \ldots, h_n) \in \bR^n\Big\}.
            \end{equation}
            In particular, it means that the horizontal bundle identifies with the vector space $\bR^n$.
        \end{proposition}
        Now let us consider two tangent vectors at $\pi(X)$ and let $U = Q_X \operatorname{diag}(h_1, \ldots, h_n) Q_X^T$ and $V = Q_X \operatorname{diag}(k_1, \ldots, h_n) Q_X^T$ be their horizontal lifts at $X$, that is their preimages in the horizontal subspace at $X$. Note that the diagonal entries of $u$ and $v$ are independent of the choice of the representative $X$. By construction, we have
        \begin{equation}\label{eq:metric}
            g_{\pi(X)}\left(d_{X}\pi\left(U\right), d_{X}\pi\left(V\right)\right) = \langle U, V \rangle_F = \sum_{i=1}^n h_i k_i
        \end{equation}
        which is non-other than the Euclidean metric of $\bR^n$ such that $\Gamma^\ast$ is isometric to the open cone $\interior(C_n)$ of $\bR^n$ as a Riemannian manifold, or in other words such that the map $\lambda$ descends to a Riemannian isometry. Another way to prove such a result is thanks to Myers-Steenrod's theorem that states that an isometry between two Riemannian manifolds is always a Riemannian isometry (see for example \citealp{petersen_riemannian_2006}). 
    
    \subsection{Geodesics and extension of the Riemannian logarithm} \label{sec:log}
        Now in particular, this result tells us that on the principal stratum, the Riemannian geodesic joining two networks identifies exactly with the geodesic segment described in Equation \ref{eq:geodesic} joining their respective spectra, and whose expression can also be derived from that of horizontal geodesics, which describe exactly the geodesics of the quotient space $\Gamma_n^\ast$ given that $\pi$ is a Riemannian submersion \citep{oneill_fundamental_1966}. In this respect, the geodesics of $\Gamma_n$ as a metric space extend the Riemannian geodesics of the main stratum to all the strata. Accordingly, we extend the Riemannian logarithm to $\Gamma_n$ as the inverse of the corresponding unique geodesic segment
        \begin{equation} \label{eq:log}
              \log_{\pi(X)}(\pi(Y)) = \lambda(Y) - \lambda(X).
        \end{equation}
        From the notion of Riemannian logarithm, one derives that of a barycenter, or weighted Fréchet mean. With such a tool now in our hands, we then move on to describing the barycentric geometry of the spectral graph spaces and setting up the barycentric subspace analysis.

\section{Barycentric subspace analysis on spectral graph spaces} \label{sec:bsa}

    \subsection{Barycentric subspaces of spectral graph spaces} \label{sec:bs}
        Essentially, barycentric subspace analysis (BSA) consists in the approximation of a collection of data points by a lower-dimensional barycentric subspace. In this section, we review the mathematical construction of BSA in the specific case of spectral graph spaces, starting with the definition of a barycentric subspace. We refer to the original work by \cite{pennec_barycentric_2018} for a general formulation of the method on a Riemannian manifold.
        
        \smallbreak
        A barycentric subspace is defined for a finite family of points as the set of all their barycenters. Concretely, the barycentric subspace of the points $\pi(A_0),\ldots, \pi(A_k) \in \Gamma_n$ is defined as
        \begin{equation}
            \operatorname{BS}(\pi(A_0),\ldots, \pi(A_k)) = \bigcup_{\substack{w_0, \ldots, w_k \in \bR \vspace{.05cm}\\ w_0 + \cdots + w_k = 1}}\bigg\{ \pi(X) \in \Gamma_n \: \Big\vert \: \sum_{i=0}^k w_i \log_{\pi(X)}(\pi(A_i)) = 0 \bigg\}.
        \end{equation}
        The points $\pi(A_0),\ldots \pi(A_k)$ are then called the reference points of the subspace. From the definition of the logarithm in \ref{eq:log}, we derive an explicit formulation of the barycentric subspaces of spectral graph spaces.
        \begin{theorem}
            \label{th:barycentric_subspace}
            Let $A_0, \ldots A_k \in \Sym(n)$. Then the barycenters of $\pi(A_0),\dots, \pi(A_k)$ identify exactly with barycenters of $\lambda(A_0),\dots, \lambda(A_k)$ and the corresponding barycentric subspace is isometric to the convex set
            \begin{multline}
                \operatorname{BS}\left(\pi(A_0),\ldots, \pi(A_k)\right) \simeq \bigg\{ \sum_{i=0}^k w_i \lambda(A_i) \: \Big\vert \: \sum_{i=0}^k w_i = 1 \:\:  \\ \textnormal{and} \:\: \sum_{i=0}^k w_i \lambda_r(A_i) \leq \sum_{i=0}^k w_i \lambda_{r+1}(A_i) \:\: \textnormal{for} \:\: 1 \leq r \leq n-1 \bigg\}.
            \end{multline}
            It is a convex polytope with at most $n-1$ facets, and it corresponds to the cross-section of the cone $C_n$ passing through the vectors $\lambda(A_0), \ldots, \lambda(A_k)$. In particular, the barycentric subspace of $\pi(A_0),\ldots, \pi(A_k)$ always contains the convex hull of $\lambda(A_0), \ldots, \lambda(A_k)$. 
        \end{theorem}
        Note that if the spectra $\lambda(A_0),\ldots, \lambda(A_k)$ are affinely independent, then the barycentric subspace of $\pi(A_0),\ldots, \pi(A_k)$ is a subspace of $\Gamma_n$ of dimension $k$. Now, one can also write the polytope $\operatorname{BS}\left(\pi(A_0),\ldots, \pi(A_k)\right)$ as the intersection of the following $n-1$ half-spaces
        \begin{equation}
            \label{eq:halfspace1}
            \begin{array}{rccccrccl}
                \alpha_{11} & \lambda_1(X) & \: + \: & \cdots & \: + \: & \alpha_{1n} & \lambda_n(X) & \: \geq \: & \beta_1 \\
                & & & & & & & \vdots & \\
                \alpha_{n-11} & \lambda_1(X) & \: + \: & \cdots & \: + \: & \alpha_{n-1n} & \lambda_n(X) & \: \geq \: & \beta_{n-1}
            \end{array}
        \end{equation}
        where for $1\leq r \leq n-1$, the pair $(\alpha_r, \beta_r)$ solves
        \begin{equation}
            \label{eq:halfspace2}
            \begin{bmatrix} \lambda_1(A_0) & \cdots & \lambda_n(A_0) & 1 \\ \vdots & & \vdots & \vdots \\ \lambda_1(A_k) & \cdots & \lambda_n(A_k) & 1 \end{bmatrix} \begin{bmatrix} \alpha_{r1} \\ \vdots \\ \alpha_{rn} \\ -\beta_r\end{bmatrix} = \begin{bmatrix} \lambda_{r+1}(A_0) - \lambda_r(A_0) \\ \vdots \\ \lambda_{r+1}(A_k) - \lambda_r(A_k) \end{bmatrix}.
        \end{equation}
        The derivations leading to \ref{eq:halfspace1} and \ref{eq:halfspace2} as well the proof of the theorem are detailed in the Supplementary material. We illustrate these results for $k=2$ in Appendix~1.
        \smallbreak
        BSA then consists in finding the barycentric subspace that minimizes the distance to the dataset $X_1, \dots, X_N$, in our case the spectral distance. Optimizing for such a subspace is a combination of the search for the reference points and the projection of the data onto the corresponding barycentric subspace. Since the latter is a projection onto a closed convex set, it has a unique minimum which can be achieved by usual gradient descent methods on $\bR^{n}$. However, without additional constraint, the optimal reference points themselves are not uniquely defined, because the same $k$-dimensional barycentric subspace can be generated by any family of $k+1$ affinely independent points of the subspace. The issue no longer arises -- with probability $1$ -- if the reference points are selected from among the data points. This approach is referred to as sample-limited barycentric subspace analysis (sample-limited BSA) in the original paper by \cite{pennec_barycentric_2018}.
    
    \subsection{Sample-limited barycentric subspace analysis} \label{sec:sample_limited_bsa}
        In this variant of BSA, the optimal reference points are restricted to be data points so that they are directly interpretable as such, following the same philosophy as in $k$-medoids clustering \citep{kaufman_partitioning_1990}. In particular, rather than reference spectra, we obtain reference networks. This in turn increases the interpretability of the low-dimensional subspace, whether used for visualization purposes or as a support for lower-dimensional analyses of the dataset. Mathematically speaking, sample-limited BSA consists of the following optimization problem
        \begin{equation}
            \label{eq:graph_bsa}
            \begin{split}
                \mini_{\substack{\vspace{.05cm}\\ 1\leq i_0, \dots, i_k \leq N \vspace{.05cm}\\ w_{ij} \in \bR \vspace{.05cm}\\ \forall 1 \leq i \leq N, \: w_{i0} + \cdots + w_{ik} = 1}} \quad & \frac{1}{N}\sum_{i=1}^N\sum_{r=1}^n \Big|\lambda_r(X_i) - \sum_{j=0}^k w_{ij} \lambda_r(X_{i_j})\Big|^2\\
                \text{subject to} \quad & \sum_{j=0}^k w_{ij} \lambda_r(X_{i_j}) < \sum_{j=0}^k w_{ij} \lambda_{r+1}(X_{i_j}) \qquad \forall i, r.  
            \end{split}
        \end{equation}
        To further increase the interpretability of the projection, we may enforce the reference networks to be extreme points of the dataset. More precisely, we restrict the projection to the convex hull of the reference networks and solve Problem \ref{eq:graph_bsa} with the additional constraint that the barycentric weights must be positive. We refer to this method as sample-limited convex BSA. It is particularly suited for the identification of archetypal points \citep{cutler_archetypal_1994}. Another aspect is that sample-limited convex BSA should be more robust to noise than standard sample-limited BSA as the volume of the convex hull of given reference networks is significantly smaller than that of their barycentric subspace (see Appendix \ref{appendix:barycentric_subspace}).

\subsection{Backward barycentric subspace analysis} \label{sec:backward}
    While searching for reference points within the sample solves the issue that optimal reference points are not uniquely defined and greatly enhances their interpretability, it nevertheless introduces a major computational challenge: the search for reference samples is a combinatorial optimization problem over all subsets of $k+1$ data points from the dataset. To circumvent this challenge, we propose to leverage the backward formulation of BSA proposed by \cite{pennec_barycentric_2018}. In the more specific sample-limited setting, backward BSA consists of finding nested subspaces $\operatorname{BS}(X_1, \dots, X_N) \supset \dots \supset \operatorname{BS}(X_{i_0}, \dots, X_{i_{k}}) \supset \cdots \supset \operatorname{BS}(X_{i_0}) =  X_{i_0}$ such that the mean squared projection error -- defined as the result of the minimization problem \ref{eq:graph_bsa} -- increases as little as possible each time a reference point is removed, starting from the barycentric subspace of the entire data set, for which the error vanishes. In this formulation, the combinatorial optimization problem is considerably reduced: instead of a single optimization problem over a set of size $\binom{N}{k+1}$ for a chosen $k$, one has to solve $N$ successive problems on sets of sizes ranging from $N$ to $2$. 
    
    \smallbreak
    In order to select the optimal number of reference points for describing the dataset, we propose searching for a break point in the mean projection error along the constructed nested sequence (see for example Figure \ref{fig:airlines-bsa}). A sudden increase in the mean squared projection error, following the removal of a reference point, suggests that the barycentric subspace prior to removal is the one best suited for representing the dataset. That is, it offers the best trade-off between accuracy and conciseness. Another approach would be to use a dimensionality estimation technique to estimate the intrinsic dimension of the dataset and thus select the barycentric subspace of that dimension within the sequence \citep{pennec_barycentric_2018}.
    
    \smallbreak
    In practice, we do not initialize backward BSA using the barycentric subspace of the entire dataset. There are several reasons for this. First, under no convexity constraint, the projection error onto a $k$-dimensional barycentric subspace vanishes for any $k$ equal to or greater than the dimension of the data -- in the case of networks, the number $n$ of nodes minus $1$ (all the networks have no self-loops and therefore a trace zero, a property preserved by the action of the orthogonal group). This means that if the size $N$ of the dataset exceeds $n-1$, the first $N-n+1$ optimization problems solved by backward BSA are superfluous, and more importantly, that the resulting data point removals are arbitrary. Under the convexity constraint, we cannot define such a clear threshold, but when $N \gg n$, we observe that the projection error remains zero for $k \gg n$. Then, if the dataset is very large, the barycentric subspace of the entire dataset and those that follow in the sequence have equally large dimensions, meaning that we cannot apply standard convex optimization techniques efficiently. Instead, we follow the suggestion by \cite{pennec_barycentric_2018} and first perform a forward BSA, another formulation of BSA as a nested construction, but which begins with a $0$-dimensional subspace -- namely, the median in the sample-limited setting, and then iteratively increases by adding at each step the reference point that minimizes the projection error onto the barycentric subspace it defines.  Under the convexity constraint, we chose to initialize forward BSA with two reference points rather than just one, since the median is not a suitable choice for an archetypal point. Note that this does not affect the complexity of the algorithm. \cite{pennec_barycentric_2018} then proposes stopping the iteration once a subspace is found for which the projection error is comparable to the level of noise in the data. Typically, we stop the iteration when either the mean squared projection error vanishes or, alternatively, when it appears to have stabilized (see for example Figure \ref{fig:eeg-bsa}. We then perform backward BSA starting from the last reference set found by forward BSA. The reason we prefer the sequence constructed by backward BSA to that constructed by forward BSA is that the forward construction constrains all reference sets in the sequence to include the one or two initial reference points found, whereas there is no reason for these to be relevant if the dataset has an intrinsic dimension larger than $2$. For example, for data grouped into three clusters, the two initial reference points found by forward BSA would tend to lie outside the clusters as much as possible \citep{pennec_barycentric_2018}, whereas one might wish for them to lie within some of the clusters instead. Letting $MSE(\{X_{i_0}, \ldots, X_{i_{k}}\})$ denote the mean squared projection error onto $BS(X_{i_0}, \ldots, X_{i_{k}})$, the algorithm proceeds as follows:
    
    \begin{algo}[Backward BSA]
    \quad
    \begin{tabbing}
    \qquad \enspace Input $X_1, \dots, X_N$ \\
    \qquad \enspace Set $k=0$ and $X_{i_0} = \argmin_{X_i} MSE(\{X_{i}\})$\\
    \qquad \enspace While $MSE(\{X_{i_0}, \ldots, X_{i_{k}}\}) > \varepsilon$ \\
    \qquad \qquad $k \leftarrow k+1$ \\
    \qquad \qquad Set $ X_{i_k} = \argmin_{X_i \neq X_{i_0}, \dots, X_{i_{k-1}}} MSE(\{X_{i_0}, \dots, X_{i_{k-1}}\} \cup \{X_i\})$ \\
    \qquad \enspace Set $ X_{i_{s_k}} = \argmin_{X_{i_s}}MSE(\{X_{i_0}, \dots, X_{i_{k}}\} \setminus \{X_{i_s}\})$ \\
    \qquad \enspace For $l=k-1$ to $l=1$ \\
    \qquad \qquad Set $ X_{i_{s_l}} = \argmin_{X_{i_s} \neq X_{i_{s_{l+1}}}, \dots, X_{i_{s_k}}}MSE(\{X_{i_0}, \dots, X_{i_{k}}\} \setminus \{X_{i_s}, X_{i_{s_{l+1}}}, \dots, X_{i_{s_k}}\})$ \\
    \qquad \enspace Set $X_{i_{s_0}} = \{X_{i_0}, \dots, X_{i_k}\} \setminus \{X_{i_{s_1}}, \dots, X_{i_{s_k}}\}$ \\
    \qquad \enspace Output $X_{i_{s_0}}, \dots, X_{i_{s_k}}$ \\
    \qquad \enspace \hspace{1.07cm} $MSE(\{X_{i_{s_0}}\}) > \cdots > MSE(\{X_{i_{s_0}}, \dots, X_{i_{s_k}}\})$
    \end{tabbing}
    \end{algo}

\section{A comparison with tangent principal component analysis} \label{sec:bsa_vs_pca}

    In this section, we present an experiment designed to underline the better interpretability of the low-dimensional subspace computed by sample-limited BSA over that computed by existing dimensionality reduction methods such as geodesic PCA \citep{calissano_populations_2024} or tangent PCA \citep{guo_quotient_2021}. We focus the comparison on tangent PCA, for it is well defined in the graph space introduced in Section \ref{sec:introduction}, and straightforward to implement. Through this experiment, we aim to show in particular that when it comes to analyzing the structural variability of networks, reference points are more intuitive quantities than vector components.
    \smallbreak
    To this end, we consider a two-dimensional set of networks featuring changes of topology, in the sense of the edge sets $E$. We generate the dataset as a two-parameter dataset $X(s_i, t_i)$ in the following fashion. The network $X(s, t)$ consists of $n=6$ nodes, two of which are always connected by a central edge of weight $s$, while the connectivity of the others depends on the value of $t$. If $t=0$, then $X(s, t)$ consists of only five edges and has a tree topology. Else, two edges of weight $|t|$ are added so that $X(s, t)$ has either a topology that we will call an \textit{eight} topology if $t<0$, or a topology that we can call a \textit{hourglass} topology if $t>0$ as illustrated in Figure \ref{fig:bsa_vs_pca-dataset}. Overall, the dataset consists in $N=16$ networks sampled from the uniform distribution over $[\frac{1}{2}, \frac{3}{2}] \times [-\frac{1}{2}, \frac{1}{2}]$. We apply sample-limited BSA and its convex variant for $k=3$ on one hand, and tangent PCA with 2 components on the other hand. In this experiment, we do not implement the forward and backward formulations of BSA because the target number of reference networks is known in advance (3), and the resulting number of combinations to explore -- 3 out of 16 -- is reasonable. Let us recall that tangent PCA is defined for a dataset on a Riemannian manifold as the PCA of the corresponding Riemannian logarithms in the tangent space at the Fréchet mean \citep{fletcher_principal_2004}. In the paper by \cite{guo_quotient_2021}, the method is redefined on the graph space introduced in Section \ref{sec:introduction} -- where the permutation matrices act instead of the orthogonal matrices. Although such a quotient space is not a manifold, one can still define a transformation analog to the logarithms by optimally aligning the networks of the dataset to their mean and then subtracting such. Note that the alignments may be computed exactly in this experiment because the networks have a small enough numbers of nodes. The two first components of tangent PCA define then a $2$-dimensional subspace. In order to interpret such a subspace, one relies in practice on the geodesics spanned by the components, in our case two straight lines of the space of symmetric matrices. As for sample-limited (convex) BSA, we base our interpretation on the three reference networks selected. 
    \begin{figure}[!h]
        \figuresize{.37}
        \figurebox{}{.\textwidth}{}[bsa_vs_pca-dataset]
        \caption{A two-parameters dataset. The networks here are drawn in such a way that the length $\ell(\{i, j\})$ of the edge $\{i, j\}$ satisfies $x_{ij} \propto \exp(-\ell(\{i, j\})^2)$. To control the edge length, we rely on the Kamada-Kawai layout available in the Python library NetworkX \citep{networkx_2008}. We keep this visualization throughout the paper and the Supplementary Material. The first parameter $s$ can be read directly as the weight of the central edge. Now if the networks represent molecules, the parameter $t$ can be interpreted as the angle between the central bond and the extreme ones. It generates three different topologies: an \textit{eight} topology, a tree topology, and a \textit{hourglass} topology (from left to right).}
        \label{fig:bsa_vs_pca-dataset}
    \end{figure}
    \smallbreak
    A first observation is that the projections output by the three methods provide planar representations of the dataset which are comparable both visually and numerically. Precisely, the squared projection error is $4,7 \times 10^{-2}$ for tangent PCA while it is $3,4 \times 10^{-2}$ for sample-limited BSA and $5,2 \times 10^{-2}$ for sample-limited convex BSA. Then, although the two first components of tangent PCA cover most variability i.e. $88\%$ of the dataset, they display in Figure \ref{fig:bsa_vs_pca-pca} rather misleading deformations of the mean network for they generate topologies that are not observed in the original dataset. This can be explained partly by the mean itself having an unobserved topology, but most importantly because the topological changes characterizing the dataset occur along non linear deformations. More generally, due to the discrete nature of topology, structural variations within of a population of networks are difficult to reproduce from continuous objects such as components. Instead, the three reference networks picked by sample-limited BSA and its convex variant (see Figures \ref{fig:bsa_vs_pca-bsa} and \ref{fig:bsa_vs_pca-convex_bsa}) recover by definition existing topologies. Additionally, we observe that the convex constraint allows to retrieve reference networks with more different structures for they are more distant to each other, while keeping projection error similarly low.
    \begin{figure}[!h]
        \figuresize{.37}
        \figurebox{}{.\textwidth}{}[bsa_vs_pca-pca]
        \caption{Tangent PCA of the two-parameters dataset. The networks spanned by the two first components (in red and blue) can be read as deformations of the Fréchet mean (framed). As expected, neither the mean nor those networks belong to the two-parameters family generating the dataset. Instead, the two deformations interpolate linearly the \textit{eight} topology and the \textit{hourglass} one, eventually resulting in negatively weighted edges (in red). As a result, they provide a misleading interpretation of the feature subspace, and so of the dataset itself, despite the projection explaining almost 90\% of the data variance.}
        \label{fig:bsa_vs_pca-pca}
    \end{figure}
    \begin{figure}[!h]
        \figuresize{.37}
        \figurebox{}{.\textwidth}{}[bsa_vs_pca-bsa]
        \caption{Sample-limited BSA of the two-parameters dataset. The barycentric coordinates thus assigned to the data points define their projection onto the plane only up to a choice of coordinates for the reference networks. We fix the position of those in such a way that their spectral distance to each other is preserved. The projected distribution is similar to that achieved with PCA, except that it appears more stretched. This is likely a consequence of the bias of PCA towards Gaussian distributions. The reference networks computed alongside the coordinates reflect the two prevailing topologies within the dataset, yet two of the three reference networks are fairly similar and therefore redundant in terms of interpretability. This issue is overcome by adding the convex constraint.}
        \label{fig:bsa_vs_pca-bsa}
    \end{figure}
    \begin{figure}[!h]
        \figuresize{.37}
        \figurebox{}{.\textwidth}{}[bsa_vs_pca-convex_bsa]
        \caption{Sample-limited convex BSA of the two-parameters dataset. Under the convex constraint, the three reference networks selected are extreme enough for us to witness three very different structure -- and in fact, almost all three possible topologies. They therefore provide a strong intuition of the data structural variability. On the other hand, the variability of the projected distribution is less than with standard sample-limited BSA. In particular, several points lie on the boundary of the triangular projection domain delimited by the three reference networks. An even more critical manifestation of this phenomenon is, for example, that the projection of the network $X_0$ is superimposed on that of the reference network $X_1$.} 
        \label{fig:bsa_vs_pca-convex_bsa}
    \end{figure}

\section{A small case study on European airlines route maps} \label{sec:airlines}

    We now apply our method to real-world datasets and illustrate how BSA can be used in practice for data visualization and exploratory data analysis. In this first case study, we look at airline route maps, that is, networks representing the volume of air traffic generated by airlines between the destinations they connect. We are interested in visualizing and characterizing the differences and similarities in the strategies of different airlines. Although the nodes of an airline route network are canonically labeled -- by the destinations served by the airline, we place ourselves in the unlabeled setting so as to compare airline strategies independently of their geographic establishment. 
    %For example, while Air France and Aeroflot concentrate their traffic in two different regions, Western Europe on the one hand and Eastern Europe on the other, we wish to highlight that the two airlines nevertheless share the same centralized strategy (see Figure \ref{fig:airlines-dataset}). 
    \smallbreak
    We consider a dataset derived from the OpenFlight database (\texttt{github.com/jpatokal/} \texttt{openflights/tree/master/data}) that consists of 12 networks representing the route maps of the main European airlines as follows: each node stands for a macro-region in and around Europe (Northern Europe, Eastern Europe, Southern Europe, Western Europe, North Africa, Western Asia) and the edge between two nodes is weighted by the proportion of routes operated between the two corresponding macro-regions out of all routes operated by the airline worldwide (see Figure \ref{fig:airlines-dataset}). Given what we have already discussed and drawing on the conclusions of the previous experiment, we limit ourselves to applying only the convex variant of sample-limited BSA. Following the strategy described in Section \ref{sec:backward}, we combine the forward and backward implementations of sample-limited convex BSA to construct a nested sequence of barycentric subspaces approximating the dataset. We stop the forward construction at the $6$-dimensional barycentric subspace, for which the projection error vanishes. Still in accordance with our strategy, we then select the smallest possible subspace from the sequence constructed backward that achieves a low mean squared projection error. 
    \begin{figure}[!h]
        \figuresize{.37}
        \figurebox{}{.\textwidth}{}[airlines-dataset]
        \caption{Route maps of 12 European airlines linking macro-regions in and around Europe. The topology of the networks reflects the strategy of the airlines and in particular the difference in this respect between low-cost airlines (three-meshed star) and most national airlines (star). }
        \label{fig:airlines-dataset}
    \end{figure}
    
    In this case, the $1$-dimensional barycentric subspace appears to be optimal. The two reference airlines that define that $1$-dimensional subspace have two opposite topologies: a sparsely connected topology (star) and a well-connected one (three-meshed star). Schematically, these two topologies exemplify the strategy of national airlines (centralized strategy) and low-cost airlines (decentralized strategy) respectively. The projection of the dataset onto the $1$-dimensional barycentric subspace can then be interpreted as the level of centralization of the airlines. 
    \begin{figure}[!h]
        \figuresize{.37}
        \figurebox{}{.\textwidth}{}[airlines-bsa]
        \caption{Sample-limited backward convex BSA of the European airlines dataset. We plot the evolution of the mean squared projection error as a function of the number of reference networks in the sequence constructed (left). The most sudden increase of the projection error occurs when jumping from 2 reference networks to 1 reference network, suggesting 1 is the lowest number of reference networks suitable for describing the dataset. The two corresponding reference networks (right) illustrate strategies that are opposite: EasyJet covers the largest number of routes, while Swiss is centralized with a focus on one particular route. Then, the projection of the dataset onto the $1$-dimensional subspace (top) ranks the airlines from Easy Jet to Swiss, and can therefore be interpreted as a measure of how equitably distributed their traffic is. Towards the center of the projected distribution, we have on one hand AirFrance and Aeroflot, implementing a centralized but more equally weighted strategy, and on the other hand Ryanair and Alitalia, for they are not strictly centralized but cover fewer routes than easyJet. In fact, Alitalia's topology looks similar to both that of EasyJet and Swiss.}
        \label{fig:airlines-bsa}
    \end{figure}

\section{Application to brain EEG functional connectivity networks} \label{sec:eeg}
    %Objectives and interest of teh dataset for teh expereiment
    Next, we look at EEG brain functional connectivity networks. EEG stands for electroencephalogram, a signal recorded by electrodes placed on a subject's scalp in order to measure the subject's brain activity. EEG functional connectivity then refers to the temporal correlation between those different signals. It describes how the different brain regions underlying the electrodes coordinate their activity. The correlation matrix of all the signals then defines a weighted graph connecting the corresponding electrodes. Now, inter-subject anatomical variability can lead to local permutations in the assignment of functions to brain regions from one subject to another \citep{calissano_graph_2024,aeschlimann_inter_2026}. Under this assumption, brain functional connectivity networks then fall under the partially labeled setting, encompassed by the unlabeled setting. For our case study, we look more precisely at the EEG database available in the UC Irvine Machine Learning Repository \citep{eeg_database_1995}. This database documents the results of a study in which control subjects and alcoholic subjects were exposed to various stimuli, with the aim of understanding the differences in cerebral activity between the two groups as they process these stimuli. We propose to explore these differences using BSA. 
    \smallbreak
    %More thorouh description of the method and the dataset
    We perform a few preprocessing steps on the raw database, detailed in the Supplementary Material. For each subject participating in the study, the preprocessed database consists of EEGs from 61 electrodes, for a total of 120 trials during which the subject underwent one of the following three paradigms: being shown one picture ($\text{C}_1$), being shown two identical pictures ($\text{C}_{\text{m}}$), or being shown two distinct pictures ($\text{C}_{\text{n}}$). The preprocessed database includes 118, 118, and 115 subjects, respectively, for each of the paradigms, of whom 43 belong to the control group in all three cases. We convert each subject's and trial's 61 EEGs into a correlation matrix, resulting in one matrix per subject and per trial. For each subject, we then average the matrices thus computed per paradigm, ultimately yielding a $61\times61$ matrix per subject and paradigm. Finally, we consider only the absolute values of the correlations: the three processed datasets (one per paradigm) therefore consists of matrices/networks of positive weights ranging from 0 to 1. Again, we apply the strategy described in Section \ref{sec:backward} and run first a forward sample-limited BSA -- of each of the three datasets -- until the mean squared projection error stabilizes: we fix a threshold of $k=19$ for all datasets. Then, we run a backward sample-limited convex BSA and select the optimal number of reference networks to describe each dataset according to how the mean squared projection error behaves along the sequence of barycentric subspaces obtained.
    \smallbreak 
    %Results
    One first observation is that, among the three paradigms, the simplest ($\text{C}_1$) is in fact the one that best separates the alcoholic group from the control group (Figure \ref{fig:eeg-bsa}). We therefore focus on the analysis of the first dataset and refer the reader to the Supplementary Material for the analysis the other two datasets. Indeed, the visualization generated by BSA for this dataset highlights, within each group a characteristic subgroup, meaning a subgroup that does not overlap with the other group. Furthermore, these two subgroups are located respectively at the two tails of the distribution as identified by to the last two reference networks in the sequence constructed by backward BSA, which embody two archetypes: a weakly connected network and a strongly connected network. In fact, the two reference networks appear to be the optimal number of reference networks for describing the dataset according to the criterion we defined (see Figure \ref{fig:eeg-bsa}). By exploring the two subgroups highlighted, we observe that the two archetypes fully characterize the subgroups (see Figure \ref{fig:eeg-subgroups}): the subgroup of healthy subjects consists exclusively of densely connected networks, whereas the subgroup of alcoholic subjects consists exclusively of weakly connected networks. And since each subgroups itself is characteristic of one group, we can conclude that high connectivity characterizes healthy subjects, and low connectivity characterizes alcoholic subjects. Finally, it should be emphasized that identifying the two subgroups simply by looking at the dataset is not straightforward (see the Supplementary Material), precisely because of the significant overlap between the two groups. 

    \begin{figure}[!h]
        \figuresize{.37}
        \figurebox{}{.\textwidth}{}[eeg_S1obj-bsa]
        \caption{Sample-limited backward convex BSA of the EEG dataset acquired in the $\text{C}_1$ paradigm. We plot the evolution of the mean squared projection error as a function of the number of reference networks in the sequence constructed (left). The most sudden increase of the projection error occurs when jumping from 3 reference networks to 2 reference networks, suggesting 2 as the lowest number of reference networks suitable for describing the dataset. The two corresponding reference networks (right) differ primarily in the intensity of their connectivity. This is also particularly visible in the heatmap visualization of the corresponding correlation matrices. Along the axis defined by the two reference networks, and ranging from networks with high connectivity to those with low connectivity, the control group (in black) and the alcohol group (in white) overlap only toward the center. At the extremes defined by each reference network, only one of the two groups is dominant: the control group toward the densely connected reference network and the alcohol group toward the weakly connected reference network, although the latter is itself part of the control group (but thus an isolated case). To visualize this phenomenon of partial overlap in 2D, we have also included the third reference network of the sequence, which is equidistant from the first two and logically corresponds to a moderately connected network.}
        \label{fig:eeg-bsa}
    \end{figure}

    \begin{figure}[!h]
        \figuresize{.37}
        \figurebox{}{.\textwidth}{}[eeg_S1obj-kde]
        \caption{Visualization of the extreme networks of each group. We highlight the separation between the two groups in the representation obtained by BSA by plotting the distribution of each group (which we estimate using the kernel density estimation function in Python). At both extremes of the total distribution, we observed that one group dominates. We select and visualize a few networks from the two well-separated subgroups thus identified. We observe that the two subgroups are well characterized by the reference network they are the closest to, and so that they differ in terms of their connectivity as much as two reference networks do.}
        \label{fig:eeg-subgroups}
    \end{figure}
    
\section{Discussion}

    Both the geometric foundations of the method presented and its utilization offer further possible developments. With regard to geometric aspects, since the spectrum of the Laplacian matrix of a network is more widely studied in the literature than that of its adjacency matrix, it would be interesting to investigate an extension of spectral graph spaces to the set of Laplacian matrices. Such an extension is however not straightforward as the space of Laplacian matrices is not closed under the action of the orthogonal group. Regarding the applicative aspects, we could further explore the interplay between the optimal reference networks defined by convex BSA and the singular points of spectral graph spaces, which often correspond to networks with a very clear structure. Indeed, on one hand, the convex constraint enforces reference networks to be extreme points of the sampled distribution. On the other hand, singular networks lie at the boundary of the spectral graph spaces. 
    % Finally, we believe our framework could provide interesting insights about real-world network-valued data such as structural brain connectivity matrices and electroencephalography (EEG) data \citep{calissano_graph_2024, sporns_graph_2018, simpson_permutation_2013}.

\section{Acknowledgments}
We would like to thank Peter Michor for the fruitful discussion on the geometric aspects discussed in this paper. The work was supported by the ERC grant $\#786854$ G-Statistics from the European Research Council, by the Chapman Fellowship scheme by the Department of Mathematics at Imperial College, by the Deutsche Forschungsgemeinschaft (DFG, German Research Foundation) under Germany's Excellence Strategy – The Berlin Mathematics Research Center MATH+ (EXC-2046/1, project ID: 390685689) and by Villum Synergy grant nr. $00069255$. The majority of this work was conducted at the Inria Centre at Université Côte d’Azur.

\section{Supplementary Material}
The Supplementary Material includes proofs of the main results, as well as details on the preprocessing of the EEG database and the results of additional simulations on that data. The code and the datasets are available at \href{https://github.com/elodiemaignant/NetBSA}{github.com/elodiemaignant/NetBSA}.

\appendix

\appendixone
\section*{Appendix 1}
\subsection*{Illustrating Theorem 3.1 for $k=2$} \label{appendix:barycentric_subspace}

    The barycentric subspace of three reference networks $\pi(A_0),\pi(A_1)$ and $\pi(A_2) \in \Gamma_n$ is a two-dimensional polytope, that is a polygon. It can be drawn isometrically in the plane in the following way. Let $a_0, a_1$ and $a_2 \in \bR^2$ be such that $\|a_i - a_j\|_{\bR^2} = \|\lambda(A_i) - \lambda(A_j)\|_{\bR^n}$. Then for $1\leq r \leq n-1$ the equation in the plane of the $r$-th facet of the polygon is 
    \begin{equation*}
        \alpha_{r1} x + \alpha_{r2} y \geq \beta_r
    \end{equation*}
        where the three coefficients $\alpha_{r1}, \alpha_{r2}, \beta_{r}$ solve
    \begin{equation*}
        \begin{bmatrix} 
            a_{01} & a_{02} & 1 \\ a_{11} & a_{12} & 1 \\ a_{21} & a_{22} & 1 
        \end{bmatrix} 
        \begin{bmatrix} 
            \alpha_{r1} \\ 
            \alpha_{r2} \\ 
            -\beta_r
        \end{bmatrix} 
            = 
        \begin{bmatrix}
            \lambda_{r+1}(A_0) - \lambda_r(A_0) \\ 
            \lambda_{r+1}(A_1) - \lambda_r(A_1) \\ 
            \lambda_{r+1}(A_2) - \lambda_r(A_2) 
        \end{bmatrix}.
    \end{equation*}
        In Figure \ref{fig:barycentric_subspace}, we visualize such a polygon for different choices of $n$ and reference networks.
    \begin{figure}[!ht]
        \figuresize{.37}
        \figurebox{}{.\textwidth}{}[barycentric_subspace]
        \caption{Barycentric subspace of 3 networks. (Top) The barycentric subspace of the 3 networks is isometric to a closed tetragon. It corresponds to the situation where the subspace has the maximum number of sides, that is one less than the number of nodes of the reference networks. (Center) The second example illustrates a situation where the barycentric subspace is isometric to a polygon that is not closed. (Bottom) In the last example, where we doubled the number of nodes, the intersection of the half-planes is not minimal and the corresponding polygon is an hexagon.}
        \label{fig:barycentric_subspace}
    \end{figure}

\appendixtwo
\section*{Appendix 2}
\subsection*{Reference points of a clustered dataset} \label{appendix:clustered}

    In an additional experiment, we investigate the case of a clustered dataset. In particular, we are interested in how the reference networks selected by sample-limited BSA localize with respect to the clusters.
    \smallbreak
    We consider a data set of size $N=15$ generated from three networks with a clearly identified topology: the star, the $2$-fully meshed star \citep{mieghem_graph_2010}, and the complete network, all with same number of nodes $n=10$. The generative model simply consists in adding Gaussian noise of standard deviation $\sigma$ to the weight of the edge between any two different nodes, eventually adjusting the resulting weight to its absolute value so that it is positive. We fix $\sigma=5 \times 10^{-2}$ such that the clusters those formed are well separated in the spectral domain. Because the spectral distance is less than the Frobenius distance for any node permutation, the clusters are then also well separated in the adjacency domain as we observe in Figure \ref{fig:clustered-dataset}. We perform a sample-limited backward BSA with and without the convex constraint (see Figure \ref{fig:airlines-bsa}). Ahead of the backward BSA, we run a forward BSA according to the strategy described in Section \ref{sec:backward}, until $k=n-1$ for the non-convex variant, and until $k=N-1$ for the convex variant (since the dataset is rather small, we can afford to run the forward BSA until the projection error vanishes, which happens only when all data points are selected).
    \begin{figure}[!h]
        \figuresize{.37}
        \figurebox{}{.\textwidth}{}[clustered-dataset]
        \caption{A data set with three clusters. The networks are clustered according to their underlying topology: star, $2$-fully meshed star or complete network (of edge weight $\frac{1}{2}$).}
        \label{fig:clustered-dataset}
    \end{figure}
    \smallbreak
    The main observation in Figure \ref{fig:clustered-dataset} is that the three last reference networks are precisely located in one cluster each, regardless of whether we implement the convex or the non-convex variant of sample-limited backward BSA. In fact, removing one of those three points results in a sudden increase of the mean squared projection error, suggesting that three is the smallest number of reference networks suitable for describing the dataset. Additionally, though sample-limited backward BSA should theoretically achieve a lower projection error than its convex variant, there is no striking discrepancy for this dataset. Still, we note that the reference networks selected by the convex variant have a more accentuated structure. While this can be attributed to them being constrained to be distant from each other as we already pointed out in the previous experiment, another possible explanation comes from the geometry of spectral graph spaces. Precisely, the three template networks are singular networks and lie on the boundary of the cone $C_{10}$. Consequently, they also lie outside of the open convex hull of the dataset such that the data points with the closer structure should be among the extreme points of the dataset.
    \begin{figure}[!h]
        \figuresize{.37}
        \figurebox{}{.\textwidth}{}[clustered-backward_bsa]
        \caption{Sample-limited backward BSA (left) and sample-limited backward convex BSA (right) of the clustered dataset. For both methods, we plot the evolution of the mean squared projection error as a function of the number of reference networks in the sequence constructed (top). The key information to take from the plots is that the projection error suddenly increases -- in ratio -- when jumping from 3 reference networks to 2 reference networks. Coincidentally, The 2-dimensional barycentric subspace is defined by one reference network per cluster (bottom).}
        \label{fig:simulated_bsa}
    \end{figure}

\bibliographystyle{biometrika}
\bibliography{biblio}

\newpage

\setcounter{section}{0}
\renewcommand{\thesection}{\arabic{section}}
\setcounter{theorem}{0}
\renewcommand{\thetheorem}{\arabic{theorem}}
\setcounter{theorem}{0}
\renewcommand{\theproposition}{\arabic{proposition}}

\begin{center}
    \textbf{ \Large Supplementary Material for Barycentric subspace analysis of network-valued data}
\end{center}

\section{Proofs}
This section provides the proofs for Proposition~\ref{prop:distance}, Proposition~\ref{prop:horizontal} and Theorem~\ref{th:barycentric_subspace} of the main manuscript.

    \begin{proposition}
        Let $X, Y \in \Sym(n)$. Then the distance between the two corresponding networks is simply the $\ell^2$ spectral distance
        \begin{equation}
            d\left( \pi(X), \pi(Y) \right) = \bigg( \sum_{i=1}^n \left|\lambda_r(X) - \lambda_r(Y)\right|^2 \bigg)^{1/2}
        \end{equation}
        where $\lambda_1(X) \leq \cdots \leq \lambda_n(X)$ and $\lambda_1(Y) \leq \cdots \leq \lambda_n(Y)$ denote the eigenvalues (with their multiplicity) of $X$ and $Y$ respectively.
    \end{proposition}
        
    \begin{proof}
        Consider the functional 
        \begin{equation}
            F_{X, Y} : R \mapsto \|R Y R^T - X \|^2.
        \end{equation}
        Its gradient at $R \in \Ortho(n)$ with respect to the Frobenius metric on $T_R\Ortho(n)=\Skew(n) \cdot R$ is
        \begin{equation}
            \grad F_{X, Y} (R) = 2 R Y R^T X - 2 X R Y R^T \in \Skew(n).
        \end{equation}
        Now the functional $F_{X, Y}$ is minimized only if the first order condition 
        \begin{equation}
            \grad F_{X, Y} (R) \in \Skew(n)^\perp = \Sym(n)
        \end{equation}
        applies. This condition is satisfied if and only if $X$ and $RYR^T$ commute with each other. Let then $R \in \Ortho(n)$ such that this is the case. Then $X$ and $RYR^T$ can be diagonalized in a common orthonormal basis
        \begin{equation}
            X = Q_{X, Y} D_X Q_{X, Y}R^T \qquad \text{and} \qquad RYR^T = Q_{X, Y} D_Y Q_{X, Y}R^T
        \end{equation}
        and we have
        \begin{equation}
            F_{X, Y}(R) = \|D_X - D_Y\|^2.
        \end{equation}
        Up to a permutation $\sigma\in \fS(n)$, this amounts exactly to
        \begin{equation}
            F_{X, Y}(R) = \sum_{r=1}^n \left\vert\lambda_r(X) - \lambda_{\sigma(r)}(Y)\right\vert^2.
        \end{equation}
        Finally, the rearrangement inequality states that
        \begin{equation}
            \sum_{r=1}^n \lambda_r \mu_{\sigma(r)} \leq \sum_{r=1}^n \lambda_r \mu_i
        \end{equation}
        for any permutation $\sigma$. Therefore, the functional $F_{X, Y}$ is minimal if $\sigma = id$ and its minimum is the quantity
        \begin{equation} \label{eq:align}
            F_{X, Y}(R) = \sum_{r=1}^n \left\vert\lambda_r(X) - \lambda_r(Y)\right\vert^2.
        \end{equation} 
        Moreover, the reader may want to check that such a choice of $\sigma$ corresponds exactly to taking $R = Q_X Q_Y^T$ where $Q_X$ and $Q_Y$ are such that $X = Q_X \operatorname{diag}(\lambda_1(X), \ldots, \lambda_n(X)) Q_X^T$ and $Y = Q_Y \operatorname{diag}(\lambda_1(Y), \ldots, \lambda_n(Y)) Q_Y^T$. Conversely, if $R = Q_X Q_Y^T$ then Equation \ref{eq:align} is satisfied.
    \end{proof}

    \begin{proposition}
        Let $X\in \Sym(n)^\ast$. Let us write $X = Q_X \operatorname{diag}(\lambda_1(X), \ldots, \lambda_n(X)) Q_X^T$ for $Q_X \in \Ortho(n)$ unique up to a sign change of each column. Then the horizontal subspace at $X$ is
        \begin{equation}
            H_X\Sym(n)^\ast = \Big\{Q_X \operatorname{diag}(h_1, \ldots, h_n) Q_X^T \: \vert \: (h_1, \ldots, h_n) \in \bR^n\Big\}.
        \end{equation}
        In particular, it means that the horizontal bundle identifies with the vector space $\bR^n$.
    \end{proposition}

    \begin{proof}
        Let $X\in \Sym(n)^\ast$. The horizontal subspace at $X$ is defined by the following construction
        \begin{equation}
            H_X \Sym(n)^\ast + \ker d_X\pi = T_X \Sym(n)^\ast.
        \end{equation}
        The subspace $\ker d_X\pi$ consists exactly of the tangent vectors at $X$ which are also tangent to the fiber $F$ of $\pi$ at $\pi(X)$. Consider then a curve $\gamma$ in the fiber $F$ such that $\gamma(0) = X$. It can be written as
        \begin{equation}
            \gamma(s) = R(s) X R(s)^T
        \end{equation}
        where $R$ is a curve in $\Ortho(n)$ with $R(0) = Id$. Then let us compute the derivative of $\gamma$ at $0$
        \begin{equation}
            \dot\gamma(0) = \dot R(0) X R(0)^T + R(0) X \dot R(0)^T.
        \end{equation}
        Since $\Ortho(n)$ is a Lie group of Lie algebra $\Skew(n)$, then there exists $A\in \Skew(n)$ such that $\dot R (0) = A$ and we have
        \begin{equation}
            \dot\gamma(0) = A X + X A^T.
        \end{equation}
        Now let $U\in H_X \Sym(n)^\ast$. Then for all $A \in \Skew(n)$ we have
        \begin{equation}
            \left \langle U, AX - XA \right \rangle = 0,
        \end{equation}
        that is
        \begin{equation}
            \left \langle XU - UX, A \right \rangle = 0.
        \end{equation}
        Since $XU - UX$ is skew-symmetric, this is equivalent to 
        \begin{equation}
            XU - UX = 0.
        \end{equation}
        Finally, if $Q_X\in \Ortho(n)$ is such that $X = Q_X \operatorname{diag}(\lambda_1(X), \ldots, \lambda_n(X)) Q_X^T$, then $U$ commutes with $X$ if and only if $Q_X^TUQ_X$ is a diagonal matrix.
    \end{proof}

    \begin{theorem}
        Let $A_0, \ldots A_k \in \Sym(n)$. Then the barycenters of $\pi(A_0),\dots, \pi(A_k)$ identify exactly with barycenters of $\lambda(A_0), \dots, \lambda(A_k)$ and the corresponding barycentric subspace is isometric to the convex set
        \begin{multline}
            \operatorname{BS}\left(\pi(A_0),\ldots, \pi(A_k)\right) \simeq \bigg\{ \sum_{i=0}^k w_i \lambda(A_i) \: \Big\vert \: \sum_{i=0}^k w_i = 1 \:\:  \\ \textnormal{and} \:\: \sum_{i=0}^k w_i \lambda_r(A_i) \leq \sum_{i=0}^k w_i \lambda_{r+1}(A_i) \:\: \textnormal{for} \:\: 1 \leq r \leq n-1 \bigg\}.
        \end{multline}
        It is a convex polytope with at most $n-1$ facets, and it corresponds to the cross-section of the cone $C_n$ passing through the vectors $\lambda(A_0), \ldots, \lambda(A_k)$. In particular, the barycentric subspace of $\pi(A_0),\ldots, \pi(A_k)$ always contains the convex hull of $\lambda(A_0), \ldots, \lambda(A_k)$. 
    \end{theorem}
    
    \begin{proof}
        Let $X \in \Sym(n)$. Then $\pi(X)$ belongs to the barycentric subspace of $\pi(A_0), \ldots, \pi(A_k)$ if and only if there exist $w_0, \ldots, w_k \in \bR$ summing to $1$ such that
        \begin{equation}
            \sum_{i=0}^k w_i \log_{\pi(X)}\left(\pi(A_i)\right) = 0.
        \end{equation}
        Following the definition of the logarithm, this condition translates directly into 
        \begin{equation}
            \sum_{i=0}^k w_i (\lambda(A_i) - \lambda(X)) = 0.
        \end{equation}
        Adding the constraint that $w_0, \ldots, w_k$ have to sum to $1$, we get
        \begin{equation} \label{eq:graph_barycenter}
            \lambda(X) = \sum_{i=0}^k w_i \lambda(A_i).
        \end{equation}
        Now we recall that by definition of the map $\lambda$, we must have $\lambda(X)\in C_n$ . Therefore, $w_0, \ldots, w_k$ should satisfy
        \begin{equation} \label{eq:graph_bs_constraints}
            \forall 1 \leq r < n, \quad \sum_{i=0}^k w_i \lambda_r(A_i) \leq \sum_{i=0}^k w_i \lambda_{r+1}(A_i).
        \end{equation}
        Let us now derive from Equations \ref{eq:graph_barycenter} and \ref{eq:graph_bs_constraints} a simple geometric constraint on $\lambda(X)$. Let us denote
        \begin{equation}
            \Lambda = 
            \begin{bmatrix} 
                \lambda_1(A_0) & \cdots & \lambda_1(A_k) \\ 
                \vdots & & \vdots \\ 
                \lambda_n(A_0) & \cdots & \lambda_n(A_k) \\ 
                1 & \cdots & 1 
            \end{bmatrix} 
            \quad \text{and} \quad \theta_r = 
            \begin{bmatrix} 
                \lambda_{r+1}(A_0) - \lambda_r(A_0) \\ 
                \vdots \\ 
                \lambda_{r+1}(A_k) - \lambda_r(A_k) 
            \end{bmatrix}
        \end{equation}
        On one hand, Equation \ref{eq:graph_bs_constraints} may be rearranged in
        \begin{equation}
            \forall 1 \leq r < n, \quad \begin{bmatrix} w_0 \\ \vdots \\ w_k \end{bmatrix}^T \theta_r \geq 0.
        \end{equation}
        On the other hand, Equation \ref{eq:graph_barycenter} together with the constraint $\sum_{i=0}^k w_i = 1$ is expressed as
        \begin{equation}
            \begin{bmatrix} \lambda_1(X) \\ \vdots \\ \lambda_n(X) \\ 1\end{bmatrix} = \Lambda \begin{bmatrix} w_0 \\ \vdots \\ w_k \end{bmatrix}.
        \end{equation}
        Notice that $\theta_r\in \ker (\Lambda)^\perp$. It follows that $\theta_r \in \mathrm{im}(\Lambda^T)$.
        For all $1 \leq r < n$, let then $\alpha_{r1}, \cdots, \alpha_{rn}$ and $\beta_r$ solve
        \begin{equation}
            \Lambda^T \begin{bmatrix} \alpha_{r1} \\ \vdots \\ \alpha_{rn} \\ - \beta_r\end{bmatrix} = \theta_r.
        \end{equation}
        Then Equation \ref{eq:graph_bs_constraints} is equivalent to $\lambda(X)$ satisfying $\alpha_r^T \lambda(X) \geq \beta_r$.
    \end{proof}

\section{Preprocessing details and complementary results for the EEG database}
This section provides the preprocessing steps for the database presented in Section~\ref{sec:eeg} of the main manuscript, alongside complementary visualizations and analyzes. The raw database includes 122 subjects. We exclude one of the alcoholic subjects from the analysis because all of the data collected for that subject are corrupted. For each subject, the database contains EEGs from 64 electrodes, from which we exclude 3, labeled X, Y, and Z, because they do not correspond to an active region of the brain and serve instead as spatial landmarks. Let us now recall that each subject completed 120 trials during which they underwent one of three possible paradigms ($\text{C}_1$, $\text{C}_{\text{m}}$, and $\text{C}_{\text{n}}$). We convert the 61 EEGs from each trial into a correlation matrix, yielding one matrix per subject and per trial, which we then average by paradigm, so that we obtain a $61\times61$ matrix, and thus a network with 61 nodes, per subject and per paradigm (see Figure~\ref{fig:eeg-dataset} for a visualization of the networks associated with the first paradigm). In a final preprocessing step, we remove a few outliers from each of the three datasets (one dataset corresponding to one paradigm) to ensure that they do not bias the results of BSA. We define these outliers as networks located at a spectral distance greater than a certain threshold (5) from more than 95\% of the other networks in the dataset. 
\smallbreak
In the main manuscript, we focus on the BSA of the first dataset. In this supplementary material, we present the results of the BSA of the other two datasets (see Figures~\ref{fig:eeg-bsa2}~and~\ref{fig:eeg-bsa3}). The key finding from analyzing the first dataset is that, within the corresponding paradigm, it is possible to identify network structures characteristic of each of the two groups -- control and alcoholic: a weakly connected network characterizes alcoholic subjects, while a densely connected network characterizes control subjects. The visualization of the other two datasets obtained via BSA is not as insightful. Specifically, within the second paradigm, while it is true that weakly connected networks correspond primarily to alcoholic subjects, it is more difficult to identify a network structure characteristic of the control group. Finally, in the third paradigm, the two groups overlap significantly in the visualization obtained via BSA, so that it provides no insight into a possible structural difference between the two groups.
\begin{figure}[!h]
        \figuresize{.37}
        \figurebox{}{.\textwidth}{}[eeg_S1obj-networks]
        \caption{Each network represents the absolute value of the average correlation matrix between the 61 electrodes placed on the scalp of the corresponding subject. The networks here are drawn in such a way that the length $\ell(\{i, j\})$ of the edge $\{i, j\}$ satisfies $x_{ij} \propto \exp(-\ell(\{i, j\})^2)$. To control the edge length, we rely on the Kamada-Kawai layout available in the Python library NetworkX \citep{networkx_2008}. Outliers identified during the preprocessing steps are included in this visualization: the fifth network in row 1, the first one in row 6, and the third one in row 11. The main observation to be made here is that it is not straightforward to differentiate the alcoholic group from the control group based solely on the visualization of the networks.}
        \label{fig:eeg-dataset}
\end{figure}
\begin{figure}[!h]
        \figuresize{.37}
        \figurebox{}{.\textwidth}{}[eeg_S2match-bsa]
        \caption{Sample-limited backward convex BSA of the EEG dataset acquired in the $\text{C}_{\text{m}}$ paradigm. We plot the evolution of the mean squared projection error as a function of the number of reference networks in the sequence constructed (left). The most sudden increase of the projection error occurs when jumping from 5 reference networks to 4 reference networks, suggesting 4 as the lowest number of reference networks suitable for describing the dataset. The corresponding representation of the data set is then three-dimensional. To visualize the dataset in 2D instead, we first perform a multidimensional scaling of the 4 reference points in 2D, and then visualize the rest of the dataset as barycenters of the 2-dimensional reference points thus generated, with the same barycentric coordinates as those obtained by BSA. Note that this visualization strategy applies more generally to any $k \geq 3$.  The first two reference networks (right) differ primarily in the intensity of their connectivity. At the extreme defined by the weakly connected reference network, the alcoholic group (in white) is dominant, such that poor connectivity characterizes alcoholic subjects, as already concluded in the first paradigm. At the extreme defined by the densely connected reference network, neither group is dominant, however. The other two reference networks selected exhibit more moderate connectivity, with one being fairly uniformly connected and the other consisting of two more strongly connected components—similar to the first reference network—which likely correspond to the two hemispheres of the brain. But this structural difference does not allow for discrimination between the two groups.}
        \label{fig:eeg-bsa2} 
\end{figure}
\begin{figure}[!h]
        \figuresize{.37}
        \figurebox{}{.\textwidth}{}[eeg_S2nomatch-bsa]
        \caption{Sample-limited backward convex BSA of the EEG dataset acquired in the $\text{C}_{\text{n}}$ paradigm. We plot the evolution of the mean squared projection error as a function of the number of reference networks in the sequence constructed (left). The most sudden increase of the projection error occurs when jumping from 4 reference networks to 3 reference networks, suggesting 3 as the lowest number of reference networks suitable for describing the dataset. Again, the first reference network differs from the second in terms of the intensity of its connectivity (right), although this difference is not as pronounced as in the first two datasets. Coincidentally, neither of these first two reference points characterizes one of the two groups. The last selected reference network is as strongly connected as the second and differs more subtly in terms of the strength of the connection between the two hemispheres identified earlier. This structural difference does not appear to characterize either of the two groups.}
        \label{fig:eeg-bsa3}
\end{figure}

\end{document}